\documentclass{article}

\usepackage{epsfig}
\usepackage{graphicx}
\usepackage{amsbsy}
\usepackage{amsmath}
\usepackage{amsfonts}
\usepackage{amssymb}
\usepackage{textcomp}
\usepackage{hyperref}
\usepackage{aliascnt}

\newcommand{\mcm}[3]{\newcommand{#1}[#2]{{\ensuremath{#3}}}} 

\mcm{\tuple}{1}{\langle #1 \rangle}
\mcm{\name}{1}{\ulcorner #1 \urcorner}
\mcm{\Nbb}{0}{\mathbb{N}}
\mcm{\Zbb}{0}{\mathbb{Z}}
\mcm{\Rbb}{0}{\mathbb{R}}
\mcm{\Cbb}{0}{\mathbb{C}}
\mcm{\Fbb}{0}{\mathbb{F}}
\mcm{\Bcal}{0}{\cal B}
\mcm{\Ccal}{0}{\cal C}
\mcm{\Dcal}{0}{\cal D}
\mcm{\Ecal}{0}{\cal E}
\mcm{\Fcal}{0}{\cal F}
\mcm{\Gcal}{0}{\cal G}
\mcm{\Hcal}{0}{\cal H}
\mcm{\Ical}{0}{\cal I}
\mcm{\Lcal}{0}{\cal L}
\mcm{\Mcal}{0}{\cal M}
\mcm{\Ncal}{0}{\cal N}
\mcm{\Pcal}{0}{{\cal P}}
\mcm{\Scal}{0}{{\cal S}}
\mcm{\Tcal}{0}{{\cal T}}
\mcm{\Ucal}{0}{{\cal U}}
\mcm{\Vcal}{0}{{\cal V}}
\mcm{\Wcal}{0}{{\cal W}}
\mcm{\Ycal}{0}{{\cal Y}}
\mcm{\Mfrak}{0}{\mathfrak M}

\mcm{\restric}{0}{\upharpoonright}
\mcm{\upset}{0}{\uparrow}
\mcm{\onto}{0}{\twoheadrightarrow}
\mcm{\smallNbb}{0}{{\small \mathbb{N}}}
\DeclareMathOperator{\preop}{op}
\mcm{\op}{0}{^{\preop}}

{\begin{array}{c}
\setlength{\unitlength}{1em}}%
{\end{array}}

\usepackage{amsthm}

\newcommand{\theoremize}[2]{\newaliascnt{#1}{thm} \newtheorem{#1}[#1]{#2} \aliascntresetthe{#1}}

\theoremstyle{plain}
\newtheorem{thm}{Theorem}[section]
\theoremize{lem}{Lemma}
\theoremize{sublem}{Sublemma}
\theoremize{claim}{Claim}
\theoremize{rem}{Remark}
\theoremize{prop}{Proposition}
\theoremize{cor}{Corollary}
\theoremize{que}{Question}
\theoremize{oque}{Open Question}
\theoremize{con}{Conjecture}

\theoremstyle{definition}
\theoremize{dfn}{Definition}
\theoremize{eg}{Example}
\theoremize{exercise}{Exercise}
\theoremstyle{plain}

\usepackage{verbatim}

\title{\scshape An excluded minors method for infinite matroids}
\author{Nathan Bowler \and Johannes Carmesin}

\begin{document}
 
\maketitle

 \begin{abstract}
The notion of thin sums matroids was invented to extend the notion of representability to non-finitary matroids. A matroid is tame if every circuit-cocircuit intersection is finite. We prove that a tame matroid is a thin sums matroid over a finite field $k$ if and only if all its finite minors are representable over $k$.

We expect that the method we use to prove this will make it possible to lift many theorems about finite matroids representable over a finite field to theorems about tame thin sums matroids over these fields. We give three examples of this: various characterisations of binary tame matroids and of regular tame matroids, and unique representability of ternary tame matroids. 
\end{abstract}

\section{Introduction}

Given a family of vectors in a vector space over some field $k$, we get a matroid structure on that family whose independent sets are given by the linearly independent subsets of the family. Matroids arising in this way are called {\em representable} matroids over $k$. A classical theorem of Tutte \cite{MR0101526} states that a finite matroid is binary (that is, representable over $\Fbb_2$) if and only if it does not have $U_{2,4}$ as a minor. In the same spirit, a key aim of finite matroid theory has been to determine such `forbidden minor' characterisations for the classes of matroids representable over other finite fields. For example Bixby and Seymour \cite{{bixby:gf3},{seymour:gf3}} characterized the finite ternary matroids (those representable over $\Fbb_3$) by forbidden minors, and more recently there is a forbidden minors characterisation for the finite matroids representable over $\Fbb_4$, due to Geelen, Gerards and Kapoor \cite{gf4}. This remains an open problem for all other finite fields. In this 
paper, we consider the 
corresponding problem for infinite matroids.

It is clear that any representable matroid is \emph{finitary}, that is, all its
circuits are finite, and so many interesting examples of infinite matroids are
not representable. However, since the construction of many standard examples,
including the algebraic cycle matroids of infinite graphs, is suggestively
similar to that of representable matroids, the notion of {\em thin sums
matroids} was introduced in \cite{RD:HB:graphmatroids}: it is a generalisation
of representability which captures these infinite examples. We will work with
thin sums matroids rather than with representable matroids.

In \cite{THINSUMS} it was shown that the class of tame thin sums matroids over a fixed field is closed under  
duality, where a matroid is \emph{tame} if any circuit-cocircuit intersection is finite.
On the other hand, there are thin sums matroids whose dual is not a thin sums matroid \cite{BC:wild_matroids} - such counterexamples cannot be tame. A simple consequence of this closure under duality is that the class of tame thin sums matroids over a fixed field is closed under taking minors, and so we may consider the forbidden minors for this class. 

Minor closed classes may have infinite `minimal' forbidden minors. For example
the class of finitary matroids has the
infinite circuit $U_{1,\Nbb}^*$ as a forbidden minor. Similarly, the class of
tame thin sums matroids over $\Rbb$ has $U_{2, \Pcal (\Rbb)}$ as a forbidden
minor. However, our main result states that the class of tame thin sums matroids
over a fixed {\em finite} field has only finite minimal forbidden minors.

\begin{thm}\label{exmin}
Let $M$ be a tame matroid and $k$ be a finite field. 
Then $M$ is a thin sums matroid over $k$ if and only if
every finite minor of $M$ is $k$-representable.
\end{thm}

The proof is by a compactness argument. All previous compactness proofs in infinite matroid theory known to the authors use only that either all finite restrictions or all finite contractions
have a certain property to conclude that the matroid itself has the desired property.
For our purposes, arguments of this kind must fail because there is a tame matroid all of whose finite restrictions and finite contractions are binary but which is not a thin sums matroid over $\Fbb_2$ - in fact, it has a $U_{2,4}$-minor. We shall briefly sketch how to construct such a matroid. Start with $U_{2,4}$, and add infinitely many edges parallel to one of its edges. This ensures that every finite contraction is binary. If we also add infinitely many edges which are parallel in the dual to some other edge then we guarantee in addition that all finite restrictions are binary, but the matroid itself has a $U_{2,4}$ minor.

Theorem \ref{exmin} implies that each of the excluded minor characterisations for finite representable matroids mentioned in the first paragraph extends to tame matroids. Thus, for example, a tame matroid is a thin sums matroid over $\Fbb_2$ if and only if it has no $U_{2,4}$ minor. Any future excluded minor characterisations for finite matroids representable over a fixed finite field will also immediately extend to tame matroids by this theorem.

We expect that our approach will make it possible to lift many other standard theorems about finite matroids representable over a finite field to theorems about tame thin sums matroids over the same field. We shall give four instances of this. Firstly, as for finite binary matroids, we obtain nine equivalent simple characterizations of tame thin sums matroids over $\Fbb_2$. Secondly, we show that a tame matroid is regular (that is, a thin sums matroid over every field) if and only if all its finite minors are, and that regularity is equivalent to signability for tame matroids (see Section 5 for a definition). Thirdly, we introduce an extension of the notion of uniqueness of representations for finite matroids to the infinite setting, and prove that tame thin sums matroids over $\Fbb_3$ have unique representations in this sense. A consequence of this is that the finite, algebraic and topological cycle matroids of a given graph each have a unique signing. We illustrate how these signings encode some of the 
structure of the graph. Fourthly, we show that the characterisations of representability over various sets of finite fields in terms of representability over partial fields extend in a uniform way to infinite tame matroids.

Our method applies to excluded minor characterisations of properties other than representability. In \cite{BCC:graphic_matroids} the same method is employed to show that the tame matroids all of whose finite minors are graphic are precisely those matroids that arise from some \emph{graph-like space}, in the sense that the circuits are given by  topological circles and the cocircuits by topological bonds.

The paper is organized as follows.
After the preliminaries in Section 2, in Section 3 we examine in detail the special case of binary matroids. Although this case is simpler than the general case, it illustrates some ideas that will be important in the proof of the main result. At the end of Section 3, we offer some related open problems.
In Section 4 we prove our main result. In Section 5 we discuss the implications
for regular and ternary matroids, and matroids representable over partial
fields.

\section{Preliminaries}

\subsection{Basics}

Throughout, notation and terminology for graphs are those of~\cite{DiestelBook10}, and for matroids those of~\cite{Oxley,matroid_axioms}.
$M$ always denotes a matroid and $E(M)$ (or just $E$), $\Ical(M)$ and $\Ccal(M)$ denote its ground 
set and its sets of independent sets and circuits, respectively.

A set system $\Ical\subseteq \Pcal(E)$ is the set of independent sets of a matroid iff it satisfies the following {\em independence  axioms\/} \cite{matroid_axioms}.
\begin{itemize}
	\item[(I1)] $\emptyset\in \Ical(M)$.
	\item[(I2)] $\Ical(M)$ is closed under taking subsets.
	\item[(I3)] Whenever $I,I'\in \Ical(M)$ with $I'$ maximal and $I$ not maximal, there exists an $x\in I'\setminus I$ such that $I+x\in \Ical(M)$.
	\item[(IM)] Whenever $I\subseteq X\subseteq E$ and $I\in\Ical(M)$, the set $\{I'\in\Ical(M)\mid I\subseteq I'\subseteq X\}$ has a maximal element.
\end{itemize}

A set system $\Ccal\subseteq \Pcal(E)$ is the set of circuits of a matroid iff it satisfies the following {\em circuit  axioms\/} \cite{matroid_axioms}.
\begin{itemize}
\item[(C1)] $\emptyset\notin\Ccal$.
\item[(C2)] No element of $\Ccal$ is a subset of another.
\item[ (C3)](Circuit elimination) Whenever $X\subseteq o\in \Ccal(M)$ and $\{o_x\mid x\in X\} \subseteq \Ccal(M)$ satisfies $x\in o_y\Leftrightarrow x=y$ for all $x,y\in X$, 
then for every $z \in o\setminus \left( \bigcup_{x \in X} o_x\right)$ there exists a  $o'\in \Ccal(M)$ such that $z\in o'\subseteq \left(o\cup  \bigcup_{x \in X} o_x\right) \setminus X$.

\item[(CM)] $\Ical$ satisfies (IM), where $\Ical$ is the set of those subsets of $E$ not including an element of $\Ccal$.
\end{itemize}


\begin{lem}\label{fdt}
 Let $M$ be a matroid and $s$ be a base.
Let $o_e$ and $b_f$ a fundamental circuit and a fundamental cocircuit with respect to $s$, then
\begin{enumerate}
 \item $o_e\cap b_f$ is empty or $o_e\cap b_f=\{e,f\}$ and
\item $f\in o_e$ iff $e\in b_f$.
\end{enumerate}
\end{lem}

\begin{proof}
To see the first note that $o_e\subseteq s+e$ and $b_f\subseteq (E\setminus s)+f$.
So $o_e\cap b_f\subseteq \{e,f\}$. As a circuit and a cocircuit can never meet in only one edge, the assertion follows.

To see the second, first let $f\in o_e$.
Then $f \in o_e \cap b_f$, so by (1) $o_e \cap b_f = \{e, f\}$ and so $e \in b_f$.
The converse implication is the dual statement of the above implication.
\end{proof}

\begin{lem}\label{o_cap_b}
 For any circuit $o$ containing two edges $e$ and $f$, there is a cocircuit $b$ such that $o\cap b=\{e,f\}$.
\end{lem}

\begin{proof}
As $o-e$ is independent, there is a base including $o-e$.
By \autoref{fdt}, the fundamental cocircuit of $f$ of this base intersects $o$ in $e$ and $f$, as desired. \end{proof}

\begin{lem} \label{rest_cir}
 Let $M$ be a matroid with ground set $E = C \dot \cup X \dot \cup D$ and let $o'$ be a circuit of $M' = M / C \backslash D$.
Then there is an $M$-circuit $o$ with $o' \subseteq o \subseteq o' \cup C$.
\end{lem}
\begin{proof}
Let $s$ be any $M$-base of $C$. Then $s \cup o'$ is $M$-dependent since $o'$ is $M'$-dependent.
On the other hand,  $s \cup o'-e$ is $M$-independent whenever $e\in o'$ since $o'-e$ is $M'$-independent.
Putting this together yields that $s \cup o'$ contains an $M$-circuit $o$, and this circuit must not avoid any $e\in o'$, as desired.
\end{proof}

\begin{cor}\label{ext_C_cap_D}
Let $M'$ be a minor of $M$. Further let $o'$ be an $M'$-circuit and $b'$ be an $M'$-cocircuit.
Then there is an $M$-circuit $o\subseteq o' \cup (E(M)\setminus E(M'))$
and an $M$-cocircuit $b\subseteq b'\cup (E(M)\setminus E(M'))$ such that $o\cap b=o'\cap b'$.
\end{cor}

\begin{proof}
 Extend $o'$ using contracted edges only and $b'$ using deleted edges only as in Lemma \ref{rest_cir} and its dual. 
\end{proof}

\subsection{Scrawls}

A \emph{scrawl} is a union of circuits. 
For any matroid $M$, $M$ can be recovered from its set of scrawls since the circuits are precisely the minimal nonempty scrawls.

We can formulate the matroid axioms in terms of scrawls. 
The scrawl axioms for a set $\Scal\subseteq \Pcal(E)$ are the following.

\begin{itemize}
	\item[(S1)] Any union of elements of $\Scal$ is in $\Scal$.
	\item[(S2)](scrawl elimination) Whenever $X\subseteq w\in \Scal$ and $\{w_x\mid x\in X\} \subseteq \Scal$ satisfies $x\in w_y\Leftrightarrow x=y$ for all $x,y\in X$, 
then for every $z \in w\setminus \left( \bigcup_{x \in X} w_x\right)$ there exists a  $w'\in \Scal$ such that $z\in w'\subseteq \left(w\cup  \bigcup_{x \in X} w_x\right) \setminus X$.

	\item[(SM)] $\Ical$ satisfies (IM), where $\Ical$ is the set of those subsets of $E$ not including a nonempty element of $\Scal$.
\end{itemize}

\begin{prop}
 Let $\Scal\subseteq \Pcal(E)$. $\Scal$ is the set of scrawls of a matroid with ground set $E$ if and only if
it satisfies the scrawl axioms.
\end{prop}

\begin{proof}
 
Clearly, (S1) and (SM) are satisfied by the set of scrawls of any matroid $M$.
Now let $w, X, \{w_x\mid x\in X\}$ and  $z$ as in (S2). 
Let $o$ be an $M$-circuit such that $z\in o\subseteq w$, which exists as $w$ is a union of circuits.
Let $\bar X= o\cap X$. For every $x\in \bar X$, let $o_x$ be an $M$-circuit such that $x\in o_x\subseteq w_x$.
Now apply circuit elimination to obtain a circuit $o'$ 
such that $z\in o'\subseteq \left(o\cup  \bigcup_{x \in \bar X} w_x\right) \setminus \bar X$.
By the choice of the $w_x$ every $x\in X\setminus \bar X$ is not in $o\cup  \bigcup_{x \in \bar X} w_x$.
This completes the proof of (S2).

Conversely, let $\Scal$ satisfy the scrawl axioms and let $\Ccal$ be the set of minimal nonempty elements of $\Scal$.
First we show that any nonempty element of $\Scal$ is a union of elements of $\Ccal$.
Let $w\in \Scal$ and $z\in w$ be given. It suffices to show that there is a minimal nonempty element $o$ of $\Scal$ such that $z\in o\subseteq w$.
By (SM) there is a maximal $I\in \Ical$ such that $I\subseteq w-z$. Let $X=(w-z)\setminus I$. For every $x\in X$, there is by maximality of $I$ an element $w_x$ of $\Scal$ 
included in $I+x$.
Now apply scrawl elimination to obtain an element $w'$ of $\Scal$ with $z\in w'\subseteq I+z$.

We show $w'\in \Ccal$. 
Suppose for a contradiction that  $w'$ properly includes a nonempty element $w''$ of $\Scal$. Note that $z\in w''$ because otherwise $I$ would include a nonempty element of $\Scal$, 
which is impossible.
Pick $x\in w'\setminus w''$.
Now we may apply scrawl elimination (with $x$ in place of $z$ and $\{z\}$ in place of $X$) to obtain an element of $\Scal$ included in $I$ which is nonempty as it contains $x$. This is a contradiction.

Second, we show that $\Ccal$ satisfies the circuit axioms.
(C1),(C2) and (CM) are clear.
Now let $o, X, \{o_x\mid x\in X\}$ and  $z$ as in (C3).
We apply (S3) and obtain a  $w'\in \Scal$ such that $z\in w'\subseteq \left(o\cup  \bigcup_{x \in X} o_x\right) \setminus X$.
Above we showed how we find an $o'\in \Ccal$ such that $z\in o'\subseteq w'$, proving (C3).

Now let $M$ be the matroid whose circuits are given by $\Ccal$. Then the set of scrawls of $M$ is $\Scal$ by (S1).
\end{proof}

Dually a \emph{coscrawl} is a union of cocircuits.
Since no circuit and cocircuit can meet in only one element,
no scrawl and coscrawl can meet in only one element. In fact, this property gives us a simple characterisation of scrawls in terms of coscrawls and vice versa.

\begin{lem}\label{is_scrawl}
Let $M$ be a matroid, and let $w\subseteq E$. The following are equivalent:
\begin{enumerate}
 \item $w$ is a scrawl of $M$.
 \item $w$ never meets a cocircuit of $M$ just once.
 \item $w$ never meets a coscrawl of $M$ just once.
\end{enumerate}
\end{lem}

\begin{proof}
It is clear that (1) implies (3) and (3) implies (2), so it suffices to show that (2) implies (1). 
Suppose that (2) holds and let $e\in w$. Then in the minor $M/(w - e)\setminus(E \setminus w)$ on the groundset $\{e\}$, $e$ cannot be a co-loop, by the dual of \autoref{rest_cir} and (2). So $e$ must be a loop, and by \autoref{rest_cir} there is a circuit $o_e$ with $e \in o_e \subseteq w$. Thus $w$ is the union of the $o_e$, and so is a scrawl.

\end{proof} 

\begin{cor}\label{rest_scrawl}
Let $M$ be a matroid with ground set $E = C \dot \cup X \dot \cup D$, and let $w' \subseteq X$. Then $w'$ is a scrawl of $M' = M / C \backslash D$ if and only if 
there is a scrawl $w$ of $M$ with $w' \subseteq w \subseteq w' \cup C$.
\end{cor}
\begin{proof}
 The `only if' direction is clear from Lemma \ref{rest_cir} and the definition of a scrawl. 
Conversely, if there is such a $w$ then by the dual of Lemma \ref{rest_cir} $w'$ can never meet a cocircuit of $M'$ just once, so by Lemma \ref{is_scrawl} it is a scrawl of $M'$.
\end{proof}

\subsection{Thin sums matroids} 

Throughout the whole paper, we will follow the convention that if we write that a sum equals zero then this implicitely includes the statement that this sum is well-defined, that is, that only finitely many summands are nonzero.

\begin{dfn}
Let $A$ be a set, and $k$ a field. Let $f = (f_e | e \in E)$ be a family of functions from $A$ to $k$, and let $\lambda = (\lambda_e | e \in E)$ be a family of elements of $k$. We say that $\lambda$ is a {\em thin dependence} of $f$ iff for each $a \in A$ we have $$\sum_{e \in E} \lambda_e f_e(a) = 0 \, ,$$

We say that a subset $I$ of $E$ is {\em thinly independent} for $f$ iff the only
thin dependence of 
$f$ which is 0 everywhere outside $I$ is $(0 | e \in E)$. The {\em thin sums
system} $M_f$ of $f$ is the set of such thinly independent sets. This isn't
always the set of independent sets of a matroid \cite{matroid_axioms}, but when
it is we call it the {\em thin sums matroid} of $f$.
\end{dfn}

This definition is deceptively similar to the definition of the representable matroid corresponding to $f$ considered as a family of vectors in the $k$-vector space $k^A$. The difference is in the more liberal definition of dependence: it is possible for $\lambda$ to be a thin dependence even if there are infinitely many $e \in E$ with $\lambda_e \neq 0$, provided that for each $a \in A$ there are only finitely many $e \in E$ such that {\em both} $\lambda_e \neq 0$ and $f_e(a) \neq 0$. 

Indeed, the notion of thin sums matroid was introduced as a generalisation of the notion of representable matroid: every representable matroid is finitary, but this restriction does not apply to thin sums matroids. 

There are many natural examples of thin sums matroids: for example, the algebraic cycle matroid of any graph not including a subdivision of the Bean graph is a thin sums matroid, as follows:

\begin{dfn}\label{rep_A}
Let $G$ be a graph with vertex set $V$ and edge set $E$, and $k$ a field. We can pick a direction for each edge $e$, calling one of its ends its {\em source} $s(e)$ and the other its {\em target} $t(e)$. Then the family $f^G = (f^G_e | e \in E)$ of functions from $V$ to $k$ is given by $f_e = \chi_{t(e)} - \chi_{s(e)}$, where for any vertex $v$ the function $\chi_v$ takes the value 1 at $v$ and 0 elsewhere.
\end{dfn}

\begin{thm}[{\cite[$\S$6.7]{matroid_axioms}}]
Let $G$ be a graph not including any subdivision of the Bean graph. Then $M_{f^G}$ is the algebraic cycle matroid of $G$.
\end{thm}

The following connection between scrawls and linear dependences will turn out to be useful.

\begin{lem}\label{thin_is_scrawl}
Let $M(f)$ be a finite representable matroid, and let $c$ be a linear dependence for $f$.
Then the support of $c$ is a scrawl.
\end{lem}

\begin{proof}
If $c$ is a linear dependence, then its support is either empty or includes a circuit $o$.
Then there is a linear dependence $c_o$ with support precisely $o$. Pick $e\in o$, then
$c-c_o\cdot \frac{c(e)}{c_o(e)}$ is also a linear dependence whose support is a proper subset of the support of $c$. If we repeat this process once for every element in the support of $c$, then we end up with the empty set and have constructed for every element in the support of $c$ a circuit that contains this element and that also is contained in the support of $c$.
\end{proof}

This Lemma is also true for infinite tame thin sums matroids \cite{THINSUMS}.

\section{Binary matroids}

\begin{thm}\label{char_binary}
Let $M$ be a tame matroid. Then the following are equivalent:
\begin{enumerate}
\item $M$ is a binary thin sums matroid.
\item For any circuit $o$ and cocircuit $b$ of $M$, $|o \cap b|$ is even.
\item For any circuit $o$ and cocircuit $b$ of $M$, $|o \cap b| \neq 3$
\item $M$ has no minor isomorphic to $U_{2,4}$. 
\item If $o_1$, $o_2$ are circuits then $o_1 \triangle o_2$ is empty or includes a circuit.
\item If $o_1$, $o_2$ are circuits then $o_1 \triangle o_2$ is a disjoint union of circuits.
\item If $(o_i | i \in I)$ is a finite family of circuits then $\bigtriangleup_{i \in I}o_i$ is empty or includes a circuit.
\item If $(o_i | i \in I)$ is a finite family of circuits then $\bigtriangleup_{i \in I}o_i$ is a disjoint union of circuits.
\item For any base $s$ of $M$, and any circuit $o$ of $M$, $o = \bigtriangleup_{e \in o \setminus s} o_e$, where $o_e$ is the fundamental circuit of $e$ with respect to $s$.
\end{enumerate}
\end{thm}

\begin{proof}
We shall prove the following implications:

 \begin{center}
 \includegraphics[height=4cm]{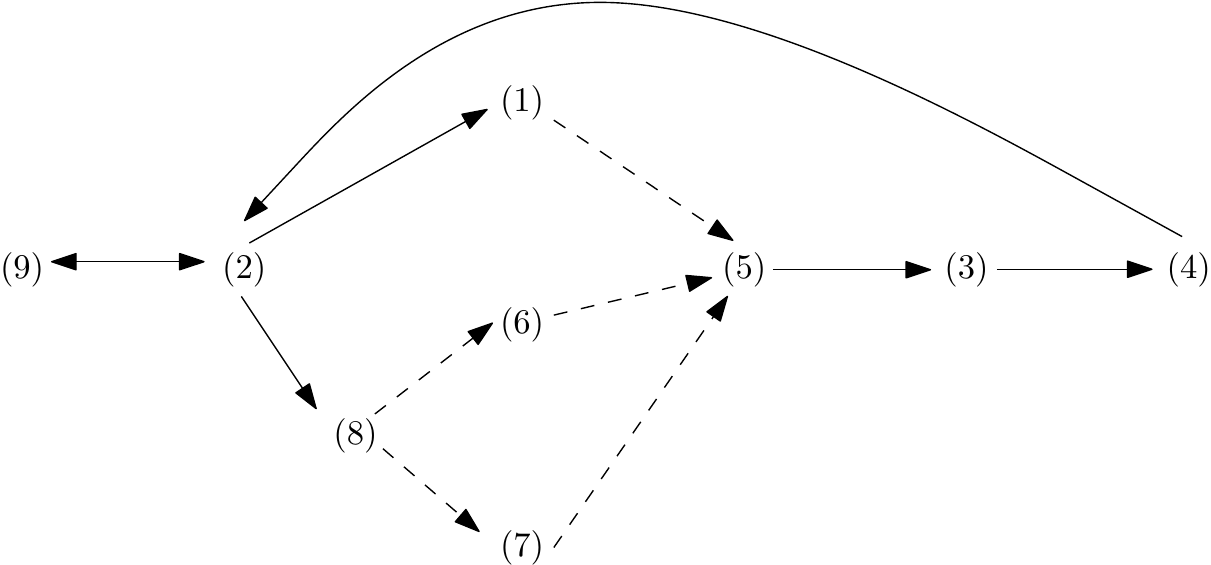}
	\label{impli}
\end{center}

Those implications indicated by dotted arrows are clear. We shall prove the remaining implications.

(2) implies (1): We need to find a suitable thin sums system. Let $A$ be the set of cocircuits of $M$, and let 
$E \xrightarrow{f} \Fbb_2^A$ 
be the map sending $e$ to the function which sends $b\in A$ to 1 if $e \in b$ and 0 otherwise.

We are to show that the thin sums matroid $M_{ts}$ defined by $f$ is $M$. 
Since the characteristic function of any $M$-circuit is a thin dependence for $f$ 
with support equal to that circuit by (2), any $M$-dependent set is also $M_{ts}(f)$-dependent.

It remains to show that the support of every non-zero thin dependence is $M$-dependent.  
By \autoref{is_scrawl} the support of every non-zero thin dependence includes a circuit, as desired.

(2) implies (8): Let $(o_i | i \in I)$ be a finite family of circuits. By Zorn's Lemma, we can choose a maximal family $(o_j | j \in J)$ 
of disjoint circuits such that $\bigcup_{j \in J} o_j \subseteq \bigtriangleup_{i \in I} o_i$, and let 
$w = \bigtriangleup_{i \in I} o_i \setminus \bigcup_{j \in J}o_j$. Let $b$ be any cocircuit of $M$, so that $|b \cap o_i|$ is even for each $i \in I$. 
Then $|b \cap \bigtriangleup_{i \in I} o_i|$ is also even, and in particular finite. Since the $o_j$ are disjoint, there can only be finitely many 
of them that meet $b \cap \bigtriangleup_{i \in I} o_i$, and since for each such $j$ we have that $|b \cap o_j|$ is even, it follows that $|b \cap w|$ is even. 
In particular, $b \cap w$ doesn't have just one element. Since $b$ was arbitrary, by \autoref{is_scrawl} $w$ is a scrawl of $M$ and so if it is nonempty it includes a circuit. 
But in that case, we could add that circuit to the family $(o_j| j \in J)$, contradicting the maximality of that family. Thus $w$ is empty, and $\
\bigtriangleup_{i \in I} o_i = \bigcup_{j \in J}o_j$ is a disjoint union of circuits.

(5) implies (3): 
Suppose, for a contradiction, that (5) holds but (3) fails, and choose a circuit $o$ and a cocircuit $b$ with $o \cap b = \{x, y, z\}$ of size 3. 
Pick a base $s$ of $(E\setminus b) + x$ including $o-y-z$, which exists by $(IM)$. As $b$ is a cocircuit,
$b - x$ avoids some $M$-base, thus $(E \setminus b) + x$ is spanning and thus $s$ is spanning, as well.
Let $o_y$ and $o_z$ be the fundamental circuits of $y$ and $z$ with respect to $s$.

It suffices to show that $o_y\triangle o_z\subseteq o-x$. Indeed, since $y,z\in o_y\triangle o_z$, (5) then yields a circuit properly included in $o$, which is impossible.
By \autoref{o_cap_b} we can't have $o_y \cap b = \{y\}$ so we must have $x \in o_y$.
Similarly, $x\in o_z$, and so $x\notin o_y\triangle o_z$.
So it is sufficient to show that $o_y$ and $o_z$ agree outside $o$, in other words: $o_y\subseteq o_z\cup o$ and $o_z\subseteq o_y\cup o$.

To see this, first note that by uniqueness of the fundamental circuit of $y$ it suffices to show that $y$ is spanned by $(o_z-z)\cup (o-y-z)$.
As $z$ is spanned by $(o_z-z)$, $o-y$ is spanned by $(o_z-z)\cup (o-y-z)$.
Since $o$ is a circuit, $y$ is also spanned by $(o_z-z)\cup (o-y-z)$, as desired. 
A similar argument yields $o_z\subseteq o_y\cup o$, completing the proof.

(3) implies (4): Since any subset of the ground set of $U_{2,4}$ of size 3 is both a circuit and a cocircuit, 
it is easy to find a circuit and cocircuit in $U_{2,4}$ whose intersection has size 3. So we simply apply \autoref{ext_C_cap_D}.

(4) implies (2): Suppose for a contradiction that (4) holds but (2) does not. Then let $o$ be a circuit and $b$ a cocircuit such that $|o \cap b| = k$
is odd. By contracting $o \setminus b$ and deleting $b \setminus o$, we obtain a minor $M'$ of $M$ in which $o \cap b$ is both a circuit and a cocircuit. 
Let $s$ be a minimal spanning set containing $o \cap b$, which exists by $(IM^*)$. 
Then in the minor $M''$ of $M'$ obtained by contracting $s \setminus (o \cap b)$, $(o \cap b)$ is spanning, and is still both a circuit and a cocircuit. 
By a similar removal, we can find a minor $M'''$ of $M''$ in which $o \cap b$ is a circuit and a cocircuit and is both spanning and cospanning. 
Let $x \in o \cap b$. Then $o \cap b - x$ is both a base and a cobase of $M'''$, and it is finite (it has size $k-1$). 
As $o\cap b-x$ is a base and a cobase, the complement of $o\cap b-x$ is also a base and a cobase.
Thus the ground set of $M'''$ is also finite (it has size $2k - 2$). 
Applying the finite version of the theorem, then, $M'''$ contains a $U_{2,4}$ minor, which is also a minor of $M$, giving the desired contradiction.

(9) implies (2): first we will show that the following implies (2):
\begin{equation}\label{9_sim}
\begin{minipage}[c]{0.8\textwidth}
For any base $s$ of $M$, any circuit $o$ meets every fundamental cocircuit of $s$ in an even number of edges.
\end{minipage}\tag{\ensuremath{\diamond}}
\end{equation}
To see that (\ref{9_sim}) implies (2), it suffices to show that every cocircuit $b$ is fundamental cocircuit of some base $s$.
Let $e\in b$. Then as $b$ is a cocircuit, $E\setminus (b-e)$ is spanning.  Thus by $(IM)$ there is a base $s$ of $E\setminus (b-e)$,
which clearly has $b$ as fundamental cocircuit.

So it remains to see that (9) implies (\ref{9_sim}).
By (9), $o = \bigtriangleup_{e \in o \setminus s} o_e$. Let $b_f$ be some fundamental cocircuit of $s$ for some $f\in s$.
By \autoref{fdt} $o_e\cap b_f$ is empty or $o_e\cap b_f=\{e,f\}$.
So it suffices to show that every $f$ is in only finitely many $o_e$,
which follows from the fact that $o = \bigtriangleup_{e \in o \setminus s} o_e$ is well defined at $f$.
This completes the proof.

(2) implies (9): we have to show for every edge $f$ that it is contained in only finitely many $o_e$
and that $ f \in o \iff f \in \bigtriangleup_{e \in o \setminus s} o_e(f)$.
If $f\notin s$, this is easy, so let $f\in s$.
By \autoref{fdt} $f\in o_e$ iff $e\in b_f$. As $M$ is tame $|o\cap b_f|$ is finite, so there are only finitely many such $e$.
By (2), $|o\cap b_f|$ is even. If $f\notin o$, all such $e$ are not contained in $s$, so $f\notin \bigtriangleup_{e \in o \setminus s} o_e$.
 If $f\in o$, all such $e$ but $f$ are not contained in $s$, so $f\in \bigtriangleup_{e \in o \setminus s} o_e$.
This completes the proof.
\end{proof}

We remark that we might also put the duals of the statements in the list onto the list.
It might be worth noting that (7) becomes false if we also allow $I$ to be infinite.
To see this, consider the finite cycle matroid of the graph obtained from a ray by adding a vertex that is adjacent to every vertex on the ray.
Indeed, the symmetric difference of all $3$-cycles is a ray starting at this new vertex.
This set is not empty, and nor does it include a circuit, so the infinite version of (7) fails.

More generally, the finite cycles of a locally finite graph generate the cycle space, which may contain infinite cycles \cite{DiestelBook10}.

We offer the following related open questions.
Let (10) be the statement like (9) but for only one base of $M$.
For finite matroids, (10) is equivalent to (9).
Is the same true for tame matroids?

The following simple question also remains open:

In \autoref{char_binary}, we assumed that $M$ is tame.
Without this assumption, the theorem is no longer true. For example, in \cite{BCP:quirkets} there is an example of a wild matroid satisfying (2-6) and (10), but not (1) or (7-9). However, this matroid is not a binary thin sums matroid. In fact, we still do not know the answer to the following:

\begin{oque}\label{Not_tame?}
Is every binary thin sums matroid tame? 
\end{oque}

In a binary tame matroid, it is easy to see that any set meeting every cocircuit not in an odd number of edges is a disjoint union of circuits provided that the set is either countable
or does not meet any cocircuit infinitely.
A well-known result of Nash-Williams says that the above is also true if the matroid is the finite cycle matroid of some graph.
Does this extend to all binary tame matroids?

\begin{oque}
Let $M$ be a binary tame matroid and let $X$ be a set that 
meets no cocircuit in an odd number of edges.
Must $X$ be a disjoint union of circuits?
\end{oque}

\section{Representable matroids}

The aim of this section is to provide an excluded-minors characterisation of thin sums matroids in the class of tame matroids. 
The following definition will be essential.

Let $k$ be a field and let $k^*$ denote the set of nonzero elements of $k$.
A \emph{$k$-painting for the matroid $M$} is a choice of a function $c_o \colon o \to k^*$
for each circuit $o$ of $M$ and a function $d_b \colon b \to k^*$ for each cocircuit $b$ of $M$ such that for any circuit $o$ and cocircuit $b$ we have
\begin{equation}\label{sumeq}
 \sum_{e \in o \cap b} c_o(e)d_b(e) = 0 \, .
\end{equation}

A matroid is \emph{$k$-paintable} if it has a $k$-painting. 
The method we will use is motivated by the following result, which is taken from \cite{THINSUMS}:

\begin{lem}\label{char}
Let $M$ be a tame matroid. Then $M$ is a thin sums matroid over the field $k$ iff 
$M$ is $k$-paintable.
\end{lem}

Proving this Lemma would go beyond the scope of this paper, but we shall not need its full power here. Instead, we shall rely on the following weaker statement:

\begin{lem}\label{mini_char}
Let $M$ be a finite matroid which is representable over the field $k$. 
Then $M$ is $k$-paintable.
\end{lem}
\begin{proof}
Let $M$ be given as $M(\phi)$ for some $E \xrightarrow{\phi} V$. For each circuit $o$ of $M$, pick some linear dependence 
$\hat{c}_o$ of $\phi$ with support $o$, and let $c_o = \hat{c}_o \restric_o$. Now let $b$ be any cocircuit of $M$, 
and fix some $e_b \in b$.
By \autoref{o_cap_b}, we can find for each $e \in b - e_b$ some circuit $o(e)$ of $M$ such that $o(e) \cap b = \{e_b, e\}$. 
We define the map $d_b \colon b \to k^*$ to be $1$ at $e_b$ and $-\frac{c_{o(e)}(e_b)}{c_{o(e)}(e)}$ for $e \in b-e_b$ 
(note that this choice ensures that \eqref{sumeq} holds for $b$ and each $o(e)$). 

Let $o$ be any circuit of 
$E$. It remains to show that $\sum_{e \in o \cap b} c_o(e)d_b(e) = 0 $.
Plugging in the values for $d_b(e)$, it remains to show
\[
 \hat{c}_o(e_b) - \sum_{e \in o \cap (b - e_b)} \frac{c_o(e)c_{o(e)}(e_b)}{c_{o(e)}(e)} =0.
\]
That is, we need $c(e_b)=0$, where 
$$c = \hat{c}_o - \sum_{e \in o \cap (b - e_b)}\frac{c_o(e)}{c_{o(e)}(e)}\hat{c}_{o(e)} \, .$$

As $c$ is a finite sum of linear dependences, it is again a linear dependence.
But for any $e \in b - e_b$, we have $c(e) = \hat{c}_o(e) -\frac{\hat{c}_o(e)}{c_{o(e)}(e)}c_{o(e)}(e) = 0$. 
By \autoref{thin_is_scrawl}, the support of $c$ is a scrawl, so it can't meet $b$ in only the point $e_b$. Thus $c(e_b)=0$, as desired.

\end{proof}

Key to this proof is the idea that the painting of the cocircuits of $M$ is determined by that of the circuits and vice versa. This fact also allows us to observe that a painting of $M$ uniquely induces paintings of all minors of $M$.

\begin{dfn}\label{indpaint}
Let $M$ be a matroid, and let $((c_o | o \in \Ccal(M)), (d_b | b \in \Ccal(M^*))$ be a $k$-painting of $M$. Let the ground set $E$ of $M$ be partitioned as $X \dot \cup C \dot \cup D$, and let $M' = M / C \backslash D$. We say that a painting $((c'_o | o \in \Ccal(M')),  (d'_b | b \in \Ccal(M'^*))$ of $M'$ is \emph{induced} by that of $M$ if and only if for each $o' \in \Ccal(M')$ there is $o \in \Ccal(M)$ with $o' \subseteq o \subseteq o' \cup C$ and such that $c'_{o'} = c_o \restric_{o'}$ and for each $b' \in \Ccal(M'^*)$ there is $b \in \Ccal(M^*)$ with $b' \subseteq b \subseteq b' \cup D$ and such that $d'_{b'} = d_b \restric_{b'}$.
\end{dfn}

It is clear from \autoref{rest_cir} that any painting of $M$ induces at least one painting of each minor $M'$. We can use the fact that the paintings of the circuits and cocircuits determine each other to show that these induced paintings are unique up to scalar factors on the $c'_o$ and $d'_b$.

\begin{lem}\label{unind}
Let $M$, $M'$ and their paintings be as in \autoref{indpaint}. Let $o'$ be any circuit of $M'$ and $o$ any circuit of $M$ with $o' \subseteq o \subseteq o' \cup C$. Then there is $\lambda \in k^*$ with $c_o \restric_{o'} = \lambda c'_{o'}$.
\end{lem}
\begin{proof}
Pick any $e \in o'$, and let $\lambda = \frac{c_o(e)}{c_{o'}'(e)}$. For any other $f \in o'$, by \autoref{o_cap_b} there is a cocircuit $b'$ of $M'$ with $o' \cap b' = \{e, f\}$. 
Since the painting of $M'$
is induced from that of $M$, there is a cocircuit $b$ of $M$ such that
$d_b(g)=d'_{b'}(g)$ for all $g\in E(M')$ and $b' \subseteq b \subseteq b' \cup
D$, and so $o \cap b = \{e,f\}$.
Using the identities in the definition of painting, we deduce that
$$c_o(e)d_{b}(e) + c_o(f)d_{b}(f) = 0 \quad\mbox{ and } \quad c'_{o'}(e)d'_{b'}(e) + c'_{o'}(f)d'_{b'}(f) = 0$$
and so 
$$c_o(f) = -\frac{c_o(e)d_{b}(e)}{d_{b}(f)} = -\frac{(\lambda c_{o'}'(e)) d'_{b'}(e)}{d'_{b'}(f)} = \lambda c'_{o'}(f)$$
which gives the desired result, since $f$ was arbitrary.
\end{proof}

Our main result is the following.

\begin{thm}\label{rep}
Let $M$ be a tame matroid and $k$ be a finite field. Then the following are equivalent.
\begin{enumerate}
 \item $M$ is a thin sums matroid over $k$.
\item $M$ is $k$-paintable.
\item Every finite minor of $M$ is $k$-representable.
\end{enumerate}
\end{thm}

In \autoref{char_binary} we already proved this theorem if $k=\Fbb_2$. The general case uses similar ideas but is more complex.

\begin{proof}
To see that (1) implies (3), we use that the class of tame thin sums matroids is closed under taking minors, as shown in \cite{THINSUMS}.

To see that (2) implies (1), let $((c_o|o\in \Ccal(M)),(d_b|b\in \Ccal(M^*)))$ be a $k$-painting.
 From this $c$ we can construct a representation of $M$ as a thin sums matroid over $k$. 
Let $E \xrightarrow{f} k^{\Ccal(M^*)}$ be the function sending $e$ to the function sending $b$ to $d_b(e)$ if $e \in b$ and to 0 otherwise. 
For any thin dependence $c$ and any cocircuit $b$ we know that $0=\sum_{e\in E} c(e) f(e)(b)= \sum_{e\in b} c(e) d_b(e)$.
In particular, the support of $c$ never meets a cocircuit only once, so it is a scrawl by \autoref{is_scrawl} and so $M$-dependent if it is nonempty.
Thus every $M_{ts}(f)$-dependent set is $M$-dependent. 

Conversely, for any $M$-circuit $o$ the function from $E$ 
to $k$ sending $e$ to $c_o(e)$ if $e \in o$ and to 0 otherwise is a thin dependence for $f$ with support equal to that circuit, 
and so any $M$-dependent set is also $M_{ts}(f)$-dependent.

It remains to show that (3) implies (2). We will use a compactness argument, so we begin by defining the topological space we will use. 
We would like an element of of this space to correspond to a choice of functions as in Lemma \ref{char} 
(though not necessarily satisfying the restrictions of that Lemma), so we take
$$H = \left(\bigcup_{o \in \Ccal(M)}\{o\} \times o\right) \amalg \left(\bigcup_{b \in \
\Ccal(M^*)}\{b\} \times b\right)$$
and take the underlying set of our space to be $X = (k^*)^H$ - the compact topology on $X$ that we will use is given by the product of $H$ copies of the discrete topology on $k^*$. 

For each circuit $o$ and cocircuit $b$ of $M$, the set 
$$C_{o, b} = \left\{c \in (k^*)^H \left| \sum_{e \in o \cap b} c(o, e)c(b, e) = 0\right.\right\}$$ is 
closed because $o \cap b$ is finite. We shall now show that any finite intersection of such sets is nonempty.

Let $K \subseteq \Ccal(M) \times \Ccal(M^*)$ be finite. Let $O$ be the set of circuits appearing as first components of elements of $K$, 
and let $B$ be the set of cocircuits appearing as second components of elements of $K$.
 Let $F = \bigcup O \cap \bigcup B$.  Note that $F$ is finite, since $M$ is tame.

Next, we shall construct a finite minor $M'$ of $M$ that will help us to prove that the finite intersection is nonempty.

\begin{lem}\label{finite_minor}
There exists a finite minor $M'$ of $M$ such that \newline
for every $o\in O$ there is an $M'$-circuit $o'\subseteq o$ such that $o'\cap F=o\cap F$ and \newline
for every $b\in B$ there is an $M'$-cocircuit $b'\subseteq b$  such that $b'\cap F=b\cap F$.
\end{lem}

\begin{proof}[Proof of the lemma]
We may assume that for each $o \in O$ and $b \in B$ the sets $o \cap F$ and $b \cap F$ are nonempty by adding an edge from $o$ or $b$ to $F$ if necessary. 
We pick an element $e_o \in o \cap F$ for each $o \in O$. Next, for each $o \in O$ and each $e \in o \cap F - e_o$ 
we pick a cocircuit $b_{o,e}$ with $o \cap b_{o,e} = \{e_o, e\}$ (this is possible by \autoref{o_cap_b}). 
Let $B'$ be the set of all cocircuits picked in this way or contained in $B$. Note that $B'$ is finite. 
Similarly, we pick for each $b \in B$ an element $e_b \in F \cap b$ and then pick a circuit $o_{b,e}$ with $o_{b,e} \cap b = \{e_b, e\}$ 
for each $e \in F \cap b - e_b$, and we collect all of these, together with all circuits contained in $O$, in a finite set $O'$.

Let $F' = \bigcup O' \cap \bigcup B'$. Note that $F'$ is also finite since $M$ is tame. Let $C = \bigcup O' \setminus F'$, 
and let $D = E \setminus \bigcup O'$. Thus $E = C \dot \cup F' \dot \cup D$. Let $M'$ be the finite minor of 
$M$ with ground set $F'$ that is given by $M /C\backslash D$. For each $o \in O$, $o \setminus F' \subseteq C$ 
and so $o \cap F'$ is a scrawl of $M'$ by \autoref{rest_scrawl}. Let $o'$ be a circuit of $M'$ with 
$e_o \in o' \subseteq o \cap F'$.  Then for each $e \in o \cap F - e_o$ we know that $F' \cap b_{o,e}$ is a 
coscrawl of $M'$, again by \autoref{rest_scrawl}, so it can't meet $o'$ in just one point. 
But $e_o \in o' \cap F' \cap b_{o,e} \subseteq \{e_o, e\}$ so we must have $o' \cap F' \cap b_{o,e} = \{e_o, e\}$ 
and we conclude that $e \in o'$. Since $e$ was arbitrary, this implies that $o \cap F \subseteq o'$.
Moreover, $o \cap F = o'\cap F$.

Similarly, for each $b \in B$, we find a cocircuit $b'$ of $M'$ such that $e_b \in b' \subseteq F' \cap b$, 
and it follows by a dual argument that $F \cap b = b' \cap F$, completing the proof of the lemma. 
\end{proof}

Since $M'$ is finite, it is representable over $k$. So we can find functions $c_o$ and $d_b$ as in \autoref{mini_char}
for this matroid. Let $c \in (k^*)^H$ be chosen so that, for each $o \in O$ and each $e \in o \cap F$ we have 
$c(o, e) = c_{o'}(e)$, and also so that for each $b \in B$ and each $e \in F \cap b$ we have $c(b, e) = c_{b'}(e)$. 
These choices ensure that $c \in \bigcap_{(o, b) \in K}C_{o, b}$.

Since $(k^*)^H$ is compact, and any finite intersection of the $C_{o,b}$ is nonempty, we have that $\bigcap_{(o, b) \in \Ccal(M) \times \Ccal(M^*)} C_{o,b}$ is nonempty.
As any element in the intersection is a $k$-painting, this completes the proof.
\end{proof}

We note that this gives a uniform way to extend excluded minor characterisations of representability from finite to infinite matroids. For example, we may immediately extend the result of \cite{{bixby:gf3},{seymour:gf3}} as follows:

\begin{cor}
A tame matroid $M$ is a thin sums matroid over $GF(3)$ if and only if it has no minor isomorphic to $U_{2,5}$, $U_{3,5}$, $F_7$ or $F_7^*$.
\end{cor}

\section{Other applications of the method}
\subsection{Regular matroids}

A key definition to prove \autoref{rep} was that of a $k$-painting.
The corresponding notion for regular matroids is as follows.

A \emph{signing} for a matroid $M$ is a choice of a function $c_o \colon o \to \{1,-1\}$
for each circuit $o$ of $M$ and a function $d_b \colon b \to \{1,-1\}$ for each cocircuit $b$ of $M$ such that for any circuit $o$ and cocircuit $b$ we have
$$\sum_{e \in o \cap b} c_o(e)d_b(e) = 0 \, ,$$
where the sum is evaluated over $\Zbb$.
A matroid is \emph{signable} if it has a signing. 

\begin{lem}\label{finsign}[{\cite[Proposition 13.4.5]{Oxley},\cite{MR921067}}]
 Let $M$ be a finite matroid.
Then $M$ is regular if and only if $M$ is signable.
\end{lem}

Using similar ideas to those in the proof of \autoref{rep}, we obtain the following.

\begin{thm}\label{reg}
Let $M$ be a tame matroid. Then the following are equivalent.
\begin{enumerate}
 \item $M$ is a thin sums matroid over every field.
\item $M$ is signable
\item Every finite minor of $M$ is regular.
\end{enumerate}
\end{thm}
\begin{proof}
(2) implies that $M$ is $k$-paintable for every field $k$, and so implies (1). (1) implies that every finite minor of $M$ is representable over every field, and so is regular, which gives (3). (3) implies that every finite minor of $M$ is signable, by \autoref{finsign}. We may then deduce (2) by a compactness argument like that in the proof of \autoref{rep}.
\end{proof}

Motivated by this theorem, we call a tame matroid \emph{regular} if any of these equivalent conditions hold.

\subsection{Partial fields}
\autoref{reg} is a special case of a more general result extending characterisations of simultaneous representations over multiple fields using partial fields to tame infinite matroids. For some background on partial fields, see \cite{Van_Zwan:thesis}.

A {\em partial field} consists of a pair $(R, S)$, where $R$ is a ring and $S$ is a subgroup of the group of units of $R$ under multiplication, such that $-1 \in S$. In this context, an \emph{$(R,S)$-painting} for a matroid $M$ is a choice of a function $c_o \colon o \to S$ for each circuit $o$ of $M$ and a function $d_b \colon b \to S$ for each cocircuit $b$ of $M$ such that for any circuit $o$ and cocircuit $b$ we have
\begin{equation}\label{sumeq}
 \sum_{e \in o \cap b} c_o(e)d_b(e) = 0 \, .
\end{equation}

For example, for any field $k$ a matroid $M$ is $k$-paintable if and only if it is $(k, k^*)$-paintable, and $M$ is signable if and only if it is $(\Zbb, \{-1, 1\})$-paintable. It is clear that the class of $(R,S)$-paintable matroids is closed under duality and under taking minors. In particular, any finite minor of an $(R,S)$-paintable matroid is $(R,S)$-paintable. The converse follows from an almost identical compactness argument to that used for \autoref{rep}, giving:

\begin{thm}
Let $(R, S)$ be a partial field with $S$ finite. A tame matroid is $(R,S)$-paintable if and only if all its finite minors are.
\end{thm}

It follows from the results of \cite[Section 2.7]{Van_Zwan:thesis} that a finite matroid is $(R,S)$-paintable if and only if it is $(R,S)$-representable. For finite matroids it is known that simultaneous representability over sets of fields corresponds to representability over partial fields, and we are now in a position to lift many such results to all tame matroids. For example, we can lift \cite[Theorem 1.2]{MR1407504} as follows:

\begin{cor}
A tame matroid $M$ is a thin sums matroid over both $\Fbb_3$ and $\Fbb_4$ if and only if it is $(\Cbb, \{\zeta^i | i \leq 6\})$-paintable for $\zeta$ a primitive sixth root of unity.
\end{cor}

\subsection{Ternary matroids}

For finite matroids, a useful property of $\Fbb_3$-representable matroids is the uniqueness of the representations. In this section, we shall prove the corresponding property for tame ternary matroids.

Let $M$ be a $k$-paintable matroid for some field $k$.
We say that two $k$-paintings $((c_o|o\in \Ccal(M)),(d_b|b\in \Ccal(M^*)))$ and 
$((\tilde c_o|o\in \Ccal(M)),(\tilde d_b|b\in \Ccal(M^*)))$ are \emph{equivalent} 
if and only if there are
constants $x(o)$ for every $o\in \Ccal(M)$, constants $x(b)$ for every $b\in \Ccal(M^*)$, constants $x(e)$ for
every edge $e$ and a field automorphism $\varphi$ such that the following are true:

\begin{enumerate}
\item $\tilde c_{o}(e)=\varphi(x(o) x(e) c_o(e))$ for any $e \in o \in \Ccal(M)$.
\item $\tilde d_{b}(e)=\varphi\left(\frac{x(b) d_b(e)}{x(e)}\right)$ for any $e \in b \in \Ccal(M^*)$.
\end{enumerate}

Two signings of the same matroid $M$ are {\em equivalent} if and only if they induce equivalent $\Fbb_3$-paintings of $M$.

Via \autoref{char} for any tame matroid any thin sums representation over $k$ corresponds to a $k$-painting. 
For finite matroids,  the notions of equivalence for representations and paintings coincide: it is straightforward to check that two representations are equivalent iff the corresponding paintings are.
As for finite matroids, we obtain the following.

\begin{thm}\label{uni}
Any two $\Fbb_3$-paintings of the same matroid $M$ are equivalent.
\end{thm}
\begin{proof}
$M$, being $\Fbb_3$-paintable, must be tame. Without loss of generality we may
also assume that $M$ is connected and has more than one edge. Thus any edges $e$
and $f$ of $M$ lie on a common circuit\footnote{In Section 3 of \cite{MR2950489}, it is shown that the relation `$e$ is in a
common circuit with $f$' is indeed an equivalence relation for infinite
matroids.}. We nominate a particular edge $g_1$, and for each other edge $g$ we
nominate a circuit $o(g)$ containing both $g_1$ and $g$. We also nominate for
each circuit $o$ of $M$ an edge $e(o) \in o$ and for each cocircuit $b$ of $M$
an edge $e(b) \in b$.

We denote the two $\Fbb_3$-paintings 
$((c_o|o\in \Ccal(M)),(d_b|b\in \Ccal(M^*)))$ and 
$((\tilde c_o|o\in \Ccal(M)),(\tilde d_b|b\in \Ccal(M^*)))$. We shall construct
witnesses to the equivalence as in the definition above.
Since every automorphism of $\Fbb_3$ is trivial, we shall take $\varphi$ to be the identity.

We now set $x(g) = \frac{\tilde c_{o(g)}(g) c_{o(g)}(g_1)}{\tilde c_{o(g)}(g_1)
c_{o(g)}(g)}$ for each $g \in E$, $x(o) = \frac{\tilde
c_o(e(o))}{x(e(o))c_o(e(o))}$ for each circuit $o$ of $M$ and $x(b) =
\frac{x(e(b)) \tilde d_b(e(b))}{d_b(e(b))}$ for each cocircuit $b$ of $M$.

In order to prove that these values satisfy (1) at a particular circuit $o$ and
$g \in o$, let $O = \{o, o(g), o(e(o))\}$ and $F = \{g, g_1, e(o)\}$ and use the
construction from the proof of \autoref{finite_minor} to obtain a finite minor
$M'=M/C\backslash D$ such that
for every $o\in O$ there is an $M'$-circuit $o'\subseteq o$ such that $o'\cap F=o\cap F$ and 
for every $b\in B$ there is an $M'$-cocircuit $b' \subseteq b$  such that $b'\cap F=b\cap F$.

Let $((c'_o | o \in \Ccal(M')), (d'_b | b \in \Ccal(M'^*))$ be the $\Fbb_3$-painting of $M'$ induced by $((c_o|o\in \Ccal(M)),(d_b|b\in \Ccal(M^*))$, and $((\tilde c'_o | o \in \Ccal(M')), (\tilde d'_b | b \in \Ccal(M'^*))$ that induced by $((\tilde c_o|o\in \Ccal(M)),(\tilde d_b|b\in \Ccal(M^*)))$.

By uniqueness of representation for finite matroids, we can find constants
$x'(o')$ for every $o'\in \Ccal(M')$, constants $x'(b')$ for every $b'\in
\Ccal(M'^*)$ and constants $x'(g)$ for
every $g \in X$ such that

\begin{enumerate}\setcounter{enumi}{2}
\item $\tilde c'_{o'}(g)=x'(o') x'(g) c'_{o'}(g)$ for any $g \in o' \in
\Ccal(M')$.
\item $\tilde d'_{b'}(g)=\frac{x'(b') d'_{b'}(g)}{x'(g)}$ for any $g \in b' \in
\Ccal(M'^*)$.
\end{enumerate}

\begin{lem}
For each $o \in O$ there is $\lambda_o \in k^*$ such that 
\begin{enumerate}\setcounter{enumi}{4}
\item $c_o \restric_F = \lambda_o c'_{o'}\restric_F$
\end{enumerate}
\end{lem}
\begin{proof}
As part of the construction of $M'$, we picked a canonical element $e_o$ of
$o'$. Let $\lambda = \frac{c_o(e_o)}{c'_{o'}(e_o)}$. For any other $e \in o'
\cap F$, there is by construction a cocircuit $b_{o,e}$ of $M$ with $o \cap
b_{o,e} = \{e_o, e\}$. Then by the dual of \autoref{rest_scrawl} $b_{o,e} \cap
E(M')$ is a coscrawl of $M'$, and so there is a cocircuit $b'$ of $M'$ with $e_o
\in b' \subseteq b_{o,e}$, and so $e_o \in o' \cap b'
\subseteq \{e_o, e\}$. Since $o'$
and $b'$ can't meet in only one element, $e \in b'$. Since the painting of $M'$
is induced from that of $M$, there is a cocircuit $b$ of $M$ such that
$d_b(e)=d'_{b'}(e)$ for all $e\in E(M')$ and $b' \subseteq b \subseteq b' \cup
D$, and so $o \cap b = \{e_o, e\}$. Using the identities in
the definition of painting, we deduce that
$$c_o(e_o)d_{b}(e_o) + c_o(e)d_{b}(e) = 0 \quad\mbox{ and } \quad c'_{o'}(e_o)d'_{b'}(e_o) + c'_{o'}(e)d'_{b'}(e) = 0$$
and so 
$$c_o(e) = -\frac{c_o(e_o)d_{b}(e_o)}{d_{b}(e)} = -\frac{\lambda c'_{o'}(e_o)
d'_{b'}(e_o)}{d'_{b'}(e)} = \lambda c'_{o'}(e)$$
which gives the desired result, since $e \in o' \cap F$ was arbitrary.
\end{proof}

Similarly, we can find constants $\tilde \lambda_o$ for each $o \in O$ such that
\begin{enumerate}\setcounter{enumi}{5}
\item $\tilde c_o \restric_{F} = \tilde \lambda_o \tilde c'_{o'}\restric_{F}$ 
\end{enumerate}

 Now we must simply unwind all the algebraic relationships to obtain the desired result.

\[
 x(g)
= \frac{\tilde c_{o(g)}(g) c_{o(g)}(g_1)}{\tilde c_{o(g)}(g_1) c_{o(g)}(g)}
= \frac{\tilde c'_{o(g)'}(g) c'_{o(g)'}(g_1)}{\tilde c'_{o(g)'}(g_1) c'_{o(g)'}(g)}
= \frac{x'(o(g)') x'(g)}{x'(o(g)')x'(g_1)}= \frac{x'(g)}{x'(g_1)}
\]

where the first equation follows from the definitions, the second from (5) and (6) and the third from (3).
Similarly, we get:

\[
 x(o)
= \frac{\tilde c_{o}(e(o))}{x(e(o))c_{o}(e(o))}
= \frac{\tilde \lambda_{o}}{\lambda_{o}} \frac{\tilde c'_{o'}(e(o))}{x(e(o))c'_{o'}(e(o))}
= \frac{\tilde \lambda_{o}}{\lambda_{o}} \frac{x'(o')x'(e(o))}{x(e(o))}
\]

And finally:

\[
x(o)x(g)c_o(g) 
= \frac{\tilde \lambda_{o}}{\lambda_{o}} \frac{x'(o')x'(e(o)) x'(g_1)}{x'(e(o))}\frac{x'(g)}{x'(g_1)}c_o(g)
= \frac{\tilde \lambda_{o}}{\lambda_{o}}x'(o')x'(g)c_o(g)
\]

Now the last term is just $\tilde c_o(g)$ by first applying (5) and then (3). This completes the proof of the above assignment satisfies (1).
The proof that it also satisfies (2) is similar.

\end{proof}
As every tame regular matroid is a thin sums matroid over $\Fbb_3$,
 it also has a unique representation.
In particular the finite cycle matroid, the algebraic cycle matroid and the topological cycle matroid of a given graph (and their duals) have a unique signing. 

In what follows, we will describe this signing of the finite cycle matroid of a given graph $G$ --- the other cases are similar. First direct the edges of $G$ in an arbitrary way.
To define the functions $c_o$, let $o$ be some cycle of $G$.
Pick a cyclic order of $o$.
For $e\in o$, let $c_o(e)=1$ if $e$ is directed according to the cyclic order of $o$ and $-1$ otherwise.

Next, let $b$ be some cocircuit. By minimality of the cocircuit, it is contained in a single component of $G$ and its removal separates this component into two components, say $C_1(b)$ and $C_2(b)$. Note that every edge in $b$ has precisely one endvertex in each of these components.
For $e\in b$, let $d_b(e)=1$ if $e$ points to a vertex in $C_1$ and $-1$ otherwise.

It remains to check that $\sum_{e\in o\cap b} c_o(e) d_b(e)=0$ for all circuits $o$ and cocircuits $b$. As every circuit is finite, the above sum is finite. Since the directions we gave to the edges of $G$ do not 
influence the values of the products $c_o(e) d_b(e)$, we may assume without loss of generality that in the bond $b$ all edges are directed from $C_1(b)$ to $C_2(b)$. So we get a summand of $+1$ for each edge along which $o$ traverses $b$ from $C_1(b)$ to $C_2(b)$ and a summand of $-1$ for each edge along which $o$ traverses $b$ from $C_2(b)$ to $C_1(b)$. Since $o$ must traverse $b$ the same number of times in each direction, the sum evaluates to 0.

Let us look at how to modify the above construction to make it work for the 
algebraic cycle matroid and the topological cycle matroid instead.
Finite circuits in the algebraic cycle matroid may be dealt with as before. To define $c_o$ for a double ray $o$, we pick an orientation of $o$ and let $c_o(e)$ be $1$ if $e$ is directed in agreement with this orientation and $-1$ otherwise. The above argument still applies: using the tameness of the algebraic cycle matroid, we obtain that a double ray can cross a skew cut only finitely many times, and both tails of the double ray must lie on the same side (as one side is rayless), so the double ray must cross the skew cut the same number of times in each direction.

Using the fact that topological circles are homeomorphic to the unit circle, we get a cyclic order on each circuit of the topological cycle matroid and the above construction again gives us a signing.
\bibliographystyle{plain}
\bibliography{literatur}

\end{document}